\documentclass[12pt]{amsart}

\usepackage[all]{xy}

\usepackage{amsmath,amssymb,latexsym, amsfonts}
\usepackage{verbatim}
\usepackage{times}
\usepackage{mathbbol}
\usepackage{hyperref}
\usepackage{mathabx}
\usepackage{bbm}
\usepackage{stmaryrd}

\def\id{{\rm id}}

\newcommand{\bract}{\blacktriangleleft}   %

\newcommand{\LMd}{\mathcal{LM}^\star}
\newcommand{\LM}{\mathcal{LM}}

\newtheorem{proposition}{Proposition}
\newtheorem{corollary}{Corollary}
\newtheorem{theorem}{Theorem}
\newtheorem{lemma}{Lemma}
\theoremstyle{definition}
\newtheorem{definition}{Definition}
\newtheorem{remark}{Remark}
\newenvironment{bemerkung}
  {\pushQED{\qed}\remark}
  {\popQED\endremark}
\newtheorem{example}{Example}
\newenvironment{beispiel}
  {\pushQED{\qed}\example}
  {\popQED\endexample}

\begin{document}
\title{Racks, Leibniz algebras and Yetter-Drinfel'd modules}
\author{Ulrich Kr\"ahmer and Friedrich Wagemann}

\address{University of Glasgow, School 
of Mathematics and Statistics, 
University Gardens, Glasgow
G12 8QW, Scotland}
\email{ulrich.kraehmer@glasgow.ac.uk}

\address{Laboratoire de Math\'ematiques Jean Leray,
Universit\'e de Nantes,
2, rue de La Houssini\`ere, 44322 Nantes cedex 3}
\email{wagemann@math.univ-nantes.fr}

\begin{abstract}
A Hopf algebra object in Loday and Pirashvili's 
category of linear maps
entails an ordinary Hopf algebra and a Yetter-Drinfel'd
module. We equip the latter with
a structure of a braided Leibniz algebra. This
provides a unified framework for examples of racks in
the category of coalgebras discussed recently 
by Carter, Crans, Elhamdadi and Saito.    
\end{abstract}
\maketitle

\tableofcontents

\section{Introduction}
The subject of the present paper is the relation between racks,
Leibniz algebras and Yetter-Drinfel'd modules. 

An augmented rack (or crossed $G$-module) 
can be defined as a Yetter-Drinfel'd
module over a group $G$, viewed as a Hopf algebra
object in
the symmetric monoidal category $({\tt Set},\times)$.  
Explicitly, it is a right $G$-set $X$ together
with a $G$-equivariant map $p : X \rightarrow G$ 
where $G$ carries the right adjoint action of $G$.
A main application of racks 
is the construction of 
invariants of links and tangles, see 
e.g.~\cite{CCES,FenRou,FreYet} and the references therein.

Leibniz algebras 
are vector spaces equipped with a
bracket that satisfies a form of the Jacobi identity, but which
is not necessarily antisymmetric, see
Definition~\ref{defi_leibni} below. 
They 
were discovered by A.M. Blokh \cite{Blo} in 1965, and then later 
rediscovered by J.-L. Loday in his
search of an understanding for the obstruction 
to periodicity in algebraic K-theory \cite{Lod}. 
In this context the problem of the integration of
Leibniz algebras arose, that is, the problem of finding
an object that is to a Leibniz algebra what a Lie group
is to its Lie algebra. Lie racks provide one possible
solution, see \cite{Cov,DheWag,Kin}. 

Analogously to augmented racks over groups, 
the Yetter-Drinfel'd modules $M$ over a Hopf algebra $H$ 
in 
$({\tt Vect},\otimes)$ 
form the Drinfel'd centre of the
monoidal category of right $H$-modules, 
see Section~\ref{section_tetra_to_YD}.
Taking 
in an $H$-tetramodule (bicovariant bimodule) 
$M$ the
invariant elements ${}^\mathrm{inv} M$ with respect to
the left coaction defines an equivalence of categories
between tetramodules and Yetter-Drinfel'd modules.   
Thus they are the coefficients in Gerstenhaber-Schack
cohomology \cite{GerSch}. 
Another application is in the classification of pointed
Hopf algebras, see e.g.~\cite{AndFanGarVen}.

Our aim here is to directly relate Leibniz algebras to
Yetter-Drinfel'd modules, starting from the fact that 
the universal enveloping algebra of a Leibniz algebra       
gives rise to a Hopf algebra object in the category $\LM$ of
linear maps \cite{LodPir}, see Section~\ref{univ_env}.
We extend some results from Wo\-ro\-no\-wicz's
theory of bicovariant differential calculi \cite{Wor} 
which are dual to Hopf algebra objects in $\LM$. 
In particular, we show that one can construct 
braided Leibniz algebras as studied by V.~Lebed
\cite{Leb} by generalising
Woronowicz's quantum Lie
algebras of finite-dimensional bicovariant differential
calculi:

\begin{theorem}\label{main}
Let $f:M\to H$ be a Hopf algebra object in the category 
of linear maps ${\mathcal L}{\mathcal M}$.
Then $f$ restricts to a morphism $\tilde f :\,^{\rm inv}M\rightarrow \mathrm{ker}\,
\varepsilon $ of Yetter-Drinfel'd modules over the Hopf
algebra $H$ and
$$
	x \lhd y=x\tilde f(y)
$$
turns   
$\,^{\rm inv}M$ into a braided Leibniz algebra in the 
category of Yetter-Drinfel'd modules.
\end{theorem}

This allows us 
to study racks and Leibniz algebras in the
same language, which provides in particular a unified 
approach to \cite[Proposition~3.1]{CCES} and
\cite[Proposition~3.5]{CCES}, see
Examples~\ref{example_Leibniz} and~\ref{beispielb} 
at the end of the paper.

The paper is structured as follows: Section~2 recalls
basic facts and definitions about the category $\LM$ of
linear maps and the construction of the universal
enveloping algebra of a Leibniz algebra. In Section~3
we explore analogues in $\LM$ of functors relating groups and
Lie algebras to Hopf algebras, with a
view towards the integration problem of Lie
algebras in $\LM$. In particular we point out that the
linearisation $p : kX \rightarrow kG$ of an augmented
rack $p : X \rightarrow G$ is not a Hopf algebra object in
$\LM$, but instead a map of $kG$-modules and comodules,
see Proposition~\ref{prop_aug_rack}. Section~4 recalls
background on Yetter-Drinfel'd modules over bialgebras.
The main section is Section~5 where we prove Theorem~\ref{main}
and finish by discussing concrete examples.\\ 

\noindent{\bf Acknowledgements:} 
UK and FW thank UC Berkeley where this work took its origin. 
FW furthermore thanks the University of Glasgow where 
this work was finalised. UK is supported by the EPSRC
Grant ``Hopf algebroids and Operads'' and the Polish
Government Grant 2012/06/M/ST1/00169.  

\section{Algebraic objects in $\LM$}  
In this section we recall the neceesary background 
on the category of linear maps, algebraic objects
therein, and the relevance of these for the theory of
Leibniz algebras, mainly from \cite{LodPir,LodPir1}.
Throughout we work with vector spaces over a field $k$,
although the results can be generalised to other
base categories. An unadorned $ \otimes $ denotes the tensor
product over $k$. 

\subsection{The tensor categories 
${\mathcal L}{\mathcal M}$ and $\LMd$}   \label{category_LM}
The following definition goes back to 
Loday and Pirashvili \cite{LodPir}:
  
\begin{definition}
The \emph{category of linear maps} $\LM$ 
has linear maps $f:V\to W$ between vector spaces 
as objects, which are usually depicted by vertical arrows 
with $V$ upstairs and $W$ downstairs. 
A morphism $\phi$ between two linear maps 
$(f:V\to W)$ and $(f':V'\to W')$ is a 
commutative square
$$\xymatrix{ V \ar[r]^{\phi_1} \ar[d]^f & V' \ar[d]^{f'} \\
W \ar[r]^{\phi_0} & W' }$$
The {\it infinitesimal tensor product} 
between $f$ and $f'$ is defined to be 
$$
	\xymatrix{ (V\otimes W')\oplus(W\otimes V') 
	\ar[d]^{f\otimes\id_{W'}+\id_W\otimes f'} \\
	W\otimes W' .}
$$
\end{definition}

The infinitesimal tensor product turns $\LM$ into a
symmetric monoidal category with 
unit object being the zero map $0:\{0\}\to k$. 

\begin{bemerkung}
Alternatively, 
$\LM$ is   
the category of $2$-term chain complexes with a truncated tensor product; one has just omitted the 
terms of degree two in the tensor product of complexes. 
One can analogously define categories 
${\mathcal L}{\mathcal M}_n$ of chain complexes of
length $n$ and a tensor
product which is truncated in degree $n$, so in this
sense 
$\LM=\LM_1$ and ${\tt Vect}=\LM_0$. 
Taking the inverse limit, one passes from
these truncated versions to the category 
of chain complexes with the ordinary tensor product
${\tt Chain}=\LM_\infty$. 
\end{bemerkung}

Interpreting $\LM$ as the category of cochain rather
than chain complexes of length 1 and depicting them
consequently by arrows pointing upwards 
results in a different
monoidal structure $\otimes^\star$ on $\LM$ in which 
$$
	(f:V\to W)\otimes^\star(f':V'\to W')
$$ 
is
given by
$$
	\xymatrix{ (V\otimes W')\oplus(W\otimes V') 
	\\
	V\otimes V' \ar[u]^{\id_V\otimes
	f'+f\otimes\id_{V'}}.}
$$
The resulting tensor category will be denoted
$\LMd$.

\subsection{Algebraic objects in $\LM$}
\label{algebraic_objects}

In a symmetric monoidal tensor category, one can define
associative algebra objects, Lie algebra objects and 
bialgebra objects. 
Loday and Pirashvili exhibit the structure of these 
in the tensor category ${\mathcal L}{\mathcal M}$.
For this, they use that the inclusion functor 
$$
	{\tt Vect}\to{\mathcal L}{\mathcal M},\quad
	W \mapsto (0:\{0\}\to W),
$$ 
and the projection functor 
$$
	\LM \to{\tt Vect},\quad
	(f:V\to W)\mapsto W
$$ 
between the categories of vector spaces ${\tt Vect}$ and 
${\mathcal L}{\mathcal M}$ are tensor functors which
compose to the identity functor on ${\tt Vect}$. 
This shows that for each of the above mentioned algebraic structures
in ${\mathcal L}{\mathcal M}$, the codomain $W$ of $ f: V
\rightarrow W$ inherits the corresponding structure in the category
of vector spaces. The linear map can be used to
turn the vector space $V \oplus W$ into 
an abelian extension of $W$, 
in the sense discussed for example in
\cite[Section~12.3.2]{LodVal}. The domain $V$ becomes 
an abelian ideal in $V \oplus W$.

More explicitly, Loday and Pirashvili 
show that in $\LM$:

\begin{itemize}
\item an associative algebra object 
$f:M\to A$ is the data of an associative algebra $A$, an 
$A$-bimodule $M$ and a bimodule map $f:M\to A$,
\item a Lie algebra object $f:M\to {\mathfrak g}$ is the data of a Lie algebra ${\mathfrak g}$,
a (right) Lie module $M$ and an equivariant map $f:M\to{\mathfrak g}$,
\item a bialgebra object $f:M\to H$ is the data of a bialgebra $H$, 
of an {\it $H$-tetramodule} (or 
\emph{bicovariant bimodule}) $M$, that is,  
an $H$-bimodule and $H$-bicomodule
whose left and right coactions are $H$-bimodule maps, 
and of an $H$-bilinear coderivation 
$f:M\to H$,
\item a Hopf algebra object in ${\mathcal L}{\mathcal M}$ is 
a bialgebra object $f:M\to H$ in ${\mathcal L}{\mathcal M}$ such that $H$ admits an antipode.      
\end{itemize}

\begin{remark}
While Loday and Pirashvili formulate their statement
about Hopf algebra objects in $\LM$ rather as a
definition, see \cite[Seciton~5.1]{LodPir}, these
really are the Hopf algebra objects in $\LM$ in the
categorical sense: it is straightforward to verify
that if $H$ has an antipode $S : H \rightarrow H$, then  
the bialgebra object $f : M \rightarrow
H$ has an antipode given by 
$$
	\xymatrix{ M \ar[r]^T \ar[d]_f & M \ar[d]^f\\ 
	H \ar[r]^S & H}
$$
with $T$ given in Sweedler notation
by $T(x)=-S(m_{(-1)})m_{(0)}S(m_{(1)})$.
Thus $T$ is uniquely determined by the antipode $S$
on $H$ and is not additional data. 
\end{remark}

\begin{bemerkung}\label{fodca}
Dually, a bialgebra object $f:H\to M$ in $\LMd$ consists of a 
bialgebra $H$ in ${\tt Vect}$ and an 
$H$-tetramodule $M$ such that $f$ is a derivation and
bicolinear. 
If 
$M=\mathrm{span}_k\{g f(h)\mid g,h \in H\}$, 
this structure is referred to as a \emph{first
order bicovariant differential
calculus} over $H$ \cite{Wor}, see
e.g.~\cite{KliSch} for a pedagogical account. Linear 
duality 
$F:V\mapsto V^*$ yields a (weakly) monoidal functor
$F : \LM \rightarrow (\LMd)^\mathrm{op}$, which is strongly
monoidal on the subcategory of finite-dimensional
vector spaces. In
Remark~\ref{fodc} below we will describe 
the class of bialgebras in $\LM$ that is under $F$ 
dual to first order bicovariant differential calculi.
\end{bemerkung}

\subsection{Universal enveloping algebras in $\LM$}
\label{univ_env}

Loday and Pirashvili furthermore construct 
in \cite{LodPir} a pair of adjoint functors 
$P$ (primitives) and 
$U$ (universal enveloping algebra) associating 
a Lie algebra object in ${\mathcal L}{\mathcal M}$
to a Hopf algebra object in ${\mathcal L}{\mathcal M}$, 
and vice versa, and prove an analogue of the classical
Milnor-Moore theorem in this context. For a given Lie algebra 
object
$f:M\to{\mathfrak g}$, the enveloping algebra is 
$\phi:U{\mathfrak g}\otimes M\to U{\mathfrak g}$,
$u \otimes m \mapsto uf(m)$.

The underlying $U \mathfrak{g} $-tetramodule structure on $U
\mathfrak{g} \otimes M$ is as follows:
the right $U{\mathfrak g}$-action on $U{\mathfrak g}\otimes M$ is induced by
$$(u\otimes m)\cdot x\,=\,ux\otimes m+u\otimes m\cdot x$$
for all $x\in{\mathfrak g}$, all $u\in U{\mathfrak g}$
and all $m\in M$.  
The left action is by multiplication on the 
left-hand factor. The left and right $U \mathfrak{g} $-coactions 
are given by the
coproduct on the left-hand factor, that is, for $x \in
\mathfrak{g} , m \in M$ they are 
$$
	(x \otimes m) \mapsto 
	1 \otimes (x \otimes m)+x \otimes (1 \otimes
m),\quad
	(x \otimes m) \mapsto 
	(1 \otimes m) \otimes x+
	(x \otimes m) \otimes 1.
$$

\subsection{Leibniz algebras}\label{laibniz}
We finally recall 
from \cite{LodPir} that 
a particular class of Lie algebra objects in $\LM$ 
arises in a canonical way from Leibniz algebras: 

\begin{definition}\label{defi_leibni}
A $k$-vector space ${\mathfrak g}$ together with a bilinear map 
$$
	[,]:{\mathfrak g}\times{\mathfrak g}\to{\mathfrak g}
$$
is called a (right) 
\emph{Leibniz algebra}, in case for all $x,y,z\in{\mathfrak g}$
$$[[x,y],z]\,=\,[x,[y,z]]+[[x,z],y]$$
holds.
\end{definition} 

In particular, any Lie algebra is a Leibniz algebra.
Conversely, for any Leibniz algebra $ \mathfrak{g} $  
the quotient by the Leibniz ideal generated 
by the squares $[x,x]$ for $x\in{\mathfrak g}$ is a 
Lie algebra $ \mathfrak{g} _\mathrm{Lie}$, and 
the right adjoint action of ${\mathfrak g}_{\rm Lie}$ on 
itself lifts to a well-defined right action on 
${\mathfrak g}$. So by construction, 
the canonical quotient map 
$ \pi : \mathfrak{g} \rightarrow
\mathfrak{g}_\mathrm{Lie}$  
is a Lie algebra object in $\LM$.
The universal enveloping algebra of $ \mathfrak{g} $ 
as defined in \cite{LodPir1} is exactly the abelian
extension of the associative algebra 
$U \mathfrak{g}_\mathrm{Lie}$ in ${\tt Vect}$ that is 
defined by the universal enveloping algebra 
$U(\mathfrak{g} \rightarrow
\mathfrak{g}_\mathrm{Lie})$, 
see \cite[Theorem~4.7]{LodPir}.

\section{The problem of integrating Lie algebras in 
${\mathcal L}{\mathcal M}$}

In this section we discuss the direct analogues in
$\LM$ of
some functorial constructions that relate
groups to Lie algebras, with 
a view towards the problem of integrating Leibniz
algebras to some global structure. 
Augmented racks and their linearisations are one
possible framework for these, so we end by recalling
some background on racks.       

\subsection{From Lie algebras to groups}
 
Consider the following diagram of functors:

\hspace{4.5cm}
\xymatrix{
{\tt Lie} \ar[r]^{U} \ar[d] & {\tt ccHopf}
\ar[d]^{-^\circ} \\ 
{\tt Grp} 
& {\tt cHopf} \ar[l]^{\chi}}
\vspace{.5cm}

Here ${\tt Lie}$ is the category of Lie algebras over the field $k$, 
${\tt Grp}$ is the category of groups, 
${\tt Hopf}$ is the category of $k$-Hopf algebras, and 
${\tt ccHopf}$ and ${\tt cHopf}$ are its subcategories
of cocommutative respectively
commutative Hopf algebras. 
The functor $U$ is that of the enveloping
algebra, and $\chi$ is the functor of
characters,
while $H^\circ$ is the Hopf dual of a Hopf algebra $H$,
that is, the Hopf algebra of matrix coefficients of
finite-dimensional representations, see
e.g.~\cite{KliSch,Mon}.

An affine algebraic group $G$ over an
algebraically closed field $k$ of characteristic 0
can be
recovered in this way from its Lie algebra 
$\mathfrak{g}:=\mathrm{Lie}(G) $
as $ \chi (U \mathfrak{g}^\circ)$ provided 
$G$ is perfect, i.e.~$G=[G,G]$. More generally,
if $G$ has unipotent radical, then $G$ is isomorphic
to the characters on the subalgebra of basic
representative functions on $U \mathfrak{g} $,    
see \cite{Hoch} for details.

\subsection{Characters of Hopf algebra objects in ${\mathcal
L}{\mathcal M}$} 
The functor $\chi(-)$ (characters) can be extended to
Hopf algebra objects 
in ${\mathcal L}{\mathcal M}$, hence one might attempt
to use it to integrate Lie algebras in $\LM$ and in
particular Leibniz algebras. 
By definition, a character $\chi$ of a Hopf algebra object $f:M\to H$ 
is an algebra morphism in ${\mathcal L}{\mathcal M}$ from $f:M\to H$ to the unit of the tensor category 
${\mathcal L}{\mathcal M}$ which is simply $0:\{0\}\to k$. This amounts to a commutative diagram 
$$\xymatrix{ M \ar[r]^{\chi_1} \ar[d]^f & \{0\} \ar[d]^{0} \\
H \ar[r]^{\chi_0} & k .}$$
One therefore obtains just characters $\chi_0$ of $H$, because $\chi_1$ is supposed to be the zero map.
The same applies to Hopf algebra objects in $\LMd$, that is,
the component of the character associated to the tetramodule vanishes. 
Thus we have:

\begin{proposition}
The functor $\chi(-)$ (characters), applied to a Hopf object in 
$\LM$ or $\LMd$, results
just in characters of the underlying Hopf algebra $H$. 
\end{proposition}

Hence the integration of Lie algebra objects 
in ${\mathcal L}{\mathcal M}$
(and thus in particular Leibniz algebras) along the lines 
outlined in the previous section must fail. 
One can associate to a Lie algebra object 
in ${\mathcal L}{\mathcal M}$ its universal 
enveloping algebra, and then by 
duality some commutative Hopf algebra object in  
$\LMd$, but characters of this object will always be only 
characters of the underlying Hopf algebra.

\subsection{Formal group laws in $\LM$}
Another approach to the integration of Lie algebras is
that of formal group laws, see \cite{Ser}. Here one
studies a continuous dual of $U \mathfrak{g}$. 

Recall that a {\it formal group law} on a vector space $V$ is
a linear map $F:S(V\oplus V)\to V$ which is unital and associative, i.e.
its extension to a coalgebra morphism $F':S(V)\otimes S(V)\to S(V)$
is an associative product on the symmetric algebra $S(V)$. 

Mostovoy \cite{Mos} transposes this definition into the realm of $\LM$. 
Namely, a formal group law in $\LM$ is a map 
$$G:S\big((V\oplus V)\to(W\oplus W)\big)\to(V\to W),$$
whose extension to a morphism of coalgebra objects
$$G':S(V\to W)\otimes S(V\to W)\to(V\to W)$$
is an algebra object in $\LM$. 
Starting with a Lie algebra object 
$M\to{\mathfrak g}$ 
 in $\LM$, the product in the universal enveloping algebra 
$U(M\to{\mathfrak g})$
composed with the projection onto the primitive 
subspace yields a formal group law
using the identification of $U(M\to{\mathfrak g})$ with $S(M\to{\mathfrak g})$ 
provided by the analogue of the 
Poincar\'e-Birkhoff-Witt theorem for Lie algebra objects in
$\LM$. Explicitly, one gets a diagram

\hspace{1.5cm}
\xymatrix{
S({\mathfrak g})\otimes M\otimes S({\mathfrak g})\oplus S({\mathfrak g})\otimes
S({\mathfrak g})\otimes M \ar[r]^{\hspace{3.3cm} G^1+G^2} \ar[d] & M \ar[d] \\ 
S({\mathfrak g})\otimes S({\mathfrak g}) \ar[r]^F & {\mathfrak g}  }
\vspace{.5cm}

Mostovoy \cite{Mos} shows then: 

\begin{proposition}
The functor that assigns to a Lie algebra object $M\to{\mathfrak g}$ in 
$\LM$ the primitive part of the product in $U(M\to{\mathfrak g})$
is an equivalence of categories of Lie algebra objects in $\LM$ and of 
formal group laws in $\LM$.
\end{proposition}

An interesting problem that
arises 
is to specify what this framework gives 
for the Lie algebra objects in $\LM$ coming from a Leibniz algebra, i.e. 
for those of the form $\pi : {\mathfrak g}\to{\mathfrak g}_{\rm Lie}$.
Furthermore, one should clarify 
what the global objects associated to these
formal group laws are.    
The results in the present paper are meant to motivate
why augmented racks are a natural candidate,
by going the other way  
and studying the Hopf algebra objects in $\LM$ that are obtained by
linearisation from augmented racks.

\subsection{Augmented racks}

The set-theoretical version of ${\mathcal L}{\mathcal M}$ is the 
category ${\mathcal M}$ of all maps $X\to Y$ between
sets $X$ and $Y$. One defines an analogue of the
infinitesimal tensor product in which 
the disjoint union of sets takes the place of the sum 
of vector spaces, and the cartesian product replaces the tensor 
product. This defines a monoidal category structure on
$\mathcal{M}$ with 
unit object $\emptyset\to\{*\}$. However, the latter is not 
terminal in ${\mathcal M}$, thus one cannot define inverses, 
and a fortiori group objects.

One way around this ``no-go'' argument is to 
consider augmented racks:

\begin{definition}
Let $X$ be a set together with a binary operation denoted $(x,y)\mapsto x\lhd y$
such that for all $y\in X$, the map $x\mapsto x\lhd y$ is bijective and
for all $x,y,z\in X$,
$$(x\lhd y)\lhd z\,=\,(x\lhd z)\lhd(y\lhd z).$$
Then we call $X$ a (right) \emph{rack}. In case the invertibility of the maps 
$x\mapsto x\lhd y$ is not required, it is called a
\emph{shelf}. 
\end{definition}

The guiding example of a rack is a group together with its conjugation map 
$(g,h)\mapsto g\lhd h:=h^{-1}gh$. 
Augmented racks are generalisations of these in which
the rack operation results from a group action: 

\begin{definition}
Let $G$ be a group and $X$ be a (right) $G$-set. 
Then a map $p:X\to G$ is called an \emph{augmented
rack} in case 
$p$ satisfies the augmentation identity, i.e. for all
$g\in G$ and all $x\in X$
\begin{equation}\label{augmentationidentity}
	p(x\cdot g)\,=\,g^{-1}\,p(x)\,g.	
\end{equation}
\end{definition}

In other words $p$ is equivariant with respect to the $G$-action on $X$ 
and the adjoint action of $G$ on itself. 
The $G$-set $X$ in an augmented rack $p:X\to G$ 
carries a canonical structure of a rack by setting
$$x\lhd y\,:=\,x\cdot p(y).$$ 

\begin{bemerkung}\label{associated_group}
Any rack $X$ can be turned into an augmented rack as
follows: let $\mathrm{As}(X)$ be the \emph{associated
group} (see for example \cite{FenRou}) of $X$, which is
the quotient of the free group on the set $X$ by the
relations $y^{-1}xy=x \lhd y$ for all $x,y \in X$.
Then there is a canonical map $p : X \rightarrow
\mathrm{As}(X)$ assigning to $x \in X$ the class of $x$
in $\mathrm{As}(X)$ which turns $X$ into an augmented
rack.   
\end{bemerkung}

A more conceptual point of view goes back to Yetter,
confer \cite{FreYet}: a group is the same as a Hopf
algebra object in the symmetric monoidal category ${\tt Set}$
with $\times$ as monoidal structure. In this sense, 
right $G$-modules 
are just right $G$-sets while right $G$-comodules are
just sets $X$ equipped with a map $p : X \rightarrow
G$. The augmentation identity
(\ref{augmentationidentity}) becomes the
Yetter-Drinfel'd condition that we will discuss in
detail in the next section. Thus augmented racks are
the same as Yetter-Drinfel'd modules over $G$ in ${\tt
Set}$, or, in other words, the category of augmented
racks over $G$ is the Drinfel'd centre of the category
of right $G$-sets.       

\subsection{Linearised augmented racks}
By linearisation, one obtains the group algebra $kG$ of  
a group $G$ which consequently is a Hopf algebra in 
${\tt Vect}$, see e.g.~\cite[p.51, Example~2]{Kas}.  
Hence one might ask whether a linearisation
of an augmented rack $p : X \rightarrow G$ defines a
Hopf algebra object in $\LM$.
The functor  $k-$ 
($k$-linearisation of a set) sends $p:X\to G$ to a 
linear map $p:kX\to kG$.
Consider $kX$ as a $kG$-bimodule where $kG$ acts on $kX$ on the 
right via the given action 
and on the left via the trivial action. 
Consider further the two linear maps 
$$
	\triangle_l : kX \rightarrow kG \otimes kX,\quad
	\triangle_r : kX \rightarrow kX \otimes kG
$$
given for $x \in X$ by
$$
\triangle_l x\,=\,p(x)\otimes x\,\,\,\,{\rm and}\,\,\,\,\,
		\triangle_r x\,=\,x\otimes p(x).
$$ 
  
Then we have: 
   
\begin{proposition}  \label{prop_aug_rack}
The maps $ \triangle_l,\triangle_r$ turn 
$kX$ into a $kG$-bicomodule
such that $p:kX\to kG$ is a morphism of bicomodules and
bimodules, where $kG$ carries the left and right
coaction given by the coproduct, the trivial left
action, and the adjoint right action.     
\end{proposition}     
 
\begin{proof}
The augmentation identity
$$
	p(x\cdot g)\,=\,g^{-1}p(x)g,\quad \forall x \in X,g
\in G
$$
shows that $p$ is a morphism of bimodules. 
We have 
$$(p\otimes 1)(\triangle_r x)\,=\,p(x)\otimes p(x)\,\,\,\,{\rm and}\,\,\,\,\,
(1\otimes p)(\triangle_l x)\,=\,p(x)\otimes p(x)$$ 
for all $x\in X$, thus $p$ is a morphism of bicomodules.
\end{proof}

In particular,  
$p : kX \rightarrow kG$ is not a Hopf algebra object in $\LM$ in
general.

\subsection{Regular functions on augmented racks}
Taking  
the coordinate ring $k[X]$ of an algebraic set $X$ 
is a contravariant functor, so applying it to an
algebraic
augmented rack $p:X\to G$ gives rise to an algebra map 
$p^* : k[G] \rightarrow k[X]$ which is most naturally
considered in $\LMd$.

  
The right $G$-action on $X$ induces a right $k[G]$-comodule
structure on $k[X]$. Together with the trivial left
comodule structure, $k[X]$
becomes a $k[G]$-bicomodule. On $k[G]$ itself, we consider the 
bicomodule structure
obtained from the trivial left coaction and 
the right adjoint coaction given in
Sweedler notation by $f \mapsto f_{(2)} \otimes
S(f_{(1)}) f_{(3)}$, and then obtain:

\begin{proposition}
$p^*:k[G]\to k[X]$ is a morphism of bimodules and bicomodules. 
\end{proposition}

\begin{proof}
For the augmented rack $p:X\to G$, we have the following commutative diagram:

\hspace{4.5cm}
\xymatrix{
X\times G \ar[r] \ar[d]^{p\times{\rm id}_G} & X \ar[d]^p\\ 
G\times G \ar[r] & G}
\vspace{.5cm}

which reads explicitly as

\hspace{3.5cm}
\xymatrix{
(x,g) \ar[r] \ar[d]^{p\times{\rm id}_G} & x\cdot g \ar[d]^p\\ 
(p(x),g) \ar[r] & p(x\cdot g)\,=\,g^{-1}p(x)g}
\vspace{.5cm}

Applying the functor $k[-]$ to this diagram yields

\hspace{4.5cm}
\xymatrix{
k[X] \ar[r]  & k[X]\otimes k[G] \\ 
k[G] \ar[u]^{p^*} \ar[r] & k[G]\otimes k[G]  \ar[u]^{p^*\otimes{\rm id}_{k[G]}}  }
\vspace{.5cm}

This means exactly that $p^*$ is a morphism of right 
comodules. As the left coactions on $k[G]$ and $k[X]$
are trivial, it is a map of bicomodules.    
\end{proof}

\subsection{The Yetter-Drinfel'd braiding}

It is well-known (see for example \cite{Kas} p. 319) that the category of 
augmented racks over a fixed group $G$ carries a braiding:

\begin{proposition} \label{braiding_augm_racks}
Define for augmented racks $p_1:X\to G$ and $p_2:Y\to G$ with respect to a fixed group $G$ 
their tensor product $X\otimes Y$ by
$X\times Y$ with the action $(x,y)\cdot g\,:=\,(x\cdot g,y\cdot g)$ and the equivariant map $p:X\times Y
\to G$ being $p(x,y):=p_1(x)p_2(y)$. Then the formula
$$c_{X,Y}:X\otimes Y\to Y\otimes X,\,\,\,\,\,\,c_{X,Y}(x,y)\,:=\,(y,x\cdot p(y))$$
defines a braiding on the category of augmented racks over $G$. 
\end{proposition}

This is just a special case of the Yetter-Drinfel'd
braiding that we are going to study in detail next.
 
\section{Yetter-Drinfel'd modules}
In this section we recall definitions and facts
about Yetter-Drinfel'd modules over Hopf algebras in
${\tt Vect}$ that we need. For more information,
the reader is referred to \cite{Kas,KliSch,Maj,Mon}. 
 
\subsection{Yetter-Drinfel'd modules}\label{section_tetra_to_YD}

Let $H=(H,\mu,\eta,\triangle,\varepsilon)$ be a bialgebra over $k$.
To every right module and right comodule $M$ over 
$H$, one functorially associates a bimodule and bicomodule $M^H$ 
over $H$ which is 
$
		H \otimes M
$
as a vector space with left and right action given by
\begin{equation*}
	g(h \otimes x):=gh \otimes x,\quad
	(h \otimes x)g:=hg_{(1)} \otimes xg_{(2)}
\end{equation*}
and left and right coaction given 
in Sweedler notation by
\begin{equation*}
	(h \otimes x)_{(-1)} \otimes 
	(h \otimes x)_{(0)} :=
	h_{(1)} \otimes (h_{(2)} \otimes x),
\end{equation*}
\begin{equation*}
	(h \otimes x)_{(0)} \otimes 
	(h \otimes x)_{(1)} :=
	(h_{(1)} \otimes x_{(0)}) \otimes 
	h_{(2)}x_{(1)}.
\end{equation*}

These coactions and actions are compatible in the sense
that 
$M^H$ is a Hopf
tetramodule if and only if $M$ 
is a Yetter-Drinfel'd module:  

\begin{definition}\label{yd_defi}
A \emph{Yetter-Drinfel'd module} 
over $H$ is a 
right module and right comodule $M$ for which we have 
\begin{equation}\label{ydcondition}
	(xh_{(2)})_{(0)} \otimes h_{(1)}(xh_{(2)})_{(1)}=
	x_{(0)} h_{(1)} \otimes 
	x_{(1)} h_{(2)}
\end{equation}
for all $x \in M$ and $h \in H$.
\end{definition}

\begin{bemerkung}
If $H$ is a Hopf algebra with antipode $S$, then 
the Yetter-Drinfel'd condition (\ref{ydcondition}) 
is easily seen to be equivalent
to 
\begin{equation}\label{ydcondition2}
	(xh)_{(0)} \otimes (xh)_{(1)}=
	x_{(0)} h_{(2)} \otimes 
	S(h_{(1)}) x_{(1)} h_{(3)}.
\end{equation}
\end{bemerkung}

More precisely, $H$ is a Hopf algebra if and only if $M \mapsto M^H$
defines an equivalence between the categories of Yetter-Drinfel'd
modules and that of Hopf tetramodules. In this case, the inverse functor is given by
taking the invariants with respect to the left coaction,
\begin{equation*}
	N \mapsto {}^\mathrm{inv}N:=\{
		x \in N \mid x_{(-1)} \otimes x_{(0)} =
	1 \otimes x\}.
\end{equation*}
This is an equivalence of monoidal categories, where the tensor
product of Hopf tetramodules is $ \otimes_H$.

\begin{example}\label{augmented_rack}
Let $G$ be a group and 
$M$ be a $kG$-Yetter-Drinfel'd module. Then $M$ is in particular a
$kG$-module, i.e. a $G$-module. The comodule structure
of $M$ is a $G$-grading
of this $G$-module:
$$M\,=\,\bigoplus_{g\in G}M_g.$$
The Yetter-Drinfel'd compatibility condition now reads for $u\in kG$ and $m\in M$ 
$$(um)_{(-1)}\otimes(um)_{(0)}\,=\,u_{(1)}m_{(-1)}S(u_{(2)})\otimes u_{(3)}m_{(0)}$$
which means for a group element $g=u\in G$ and a homogeneous element $m\in M_h$
$$(gm)_{(-1)}\otimes(gm)_{(0)}\,=\,ghg^{-1}\otimes g\cdot m.$$
This means that the action of $g \in G$ on $M$ maps $M_h$ 
to  $M_{ghg^{-1}}$.

When the module $M$ is a permutation representation of
$G$, that is, is obtained by linearisation from a
(right) $G$-set $X$, $M \simeq kX$, then $M$ is
Yetter-Drinfel'd precisely when $X$ carries the
structure of an augmented rack. The full subcategory of
the category of all Yetter-Drinfel'd modules over $kG$
of these permutation modules has been studied first by
Freyd and Yetter, see  
\cite[Definition 4.2.3]{FreYet}.
\end{example}

\begin{example}\label{einhuellende}
Recall from Section~\ref{univ_env} that if 
$ f : M \rightarrow \mathfrak{g} $ is any Lie
algebra object in $\LM$, then the universal enveloping algebra
construction in $\LM$ yields the $U \mathfrak{g} $-tetramodule 
$U \mathfrak{g} \otimes M$. In this case, $M$ is
recovered as the Yetter-Drinfel'd module of left
invariant elements, with trivial right coaction and
right action being induced by the right $
\mathfrak{g} $-module structure on $M$.

More generally, every right module over a cocommutative
bialgebra $H$ becomes a Yetter-Drinfel'd module with
respect to the trivial right coaction.   
\end{example}

\subsection{The Yetter-Drinfel'd braiding revisited} \label{section_braidings}
Every right $H$-module and right $H$-comodule $M$ carries a
canonical map 
\begin{equation}\label{ydbraiding}
		\tau : M \otimes M \rightarrow M \otimes M,\quad
		x \otimes y \mapsto y_{(0)} \otimes xy_{(1)}
\end{equation}
The following well-known fact characterises when $ \tau $ is
a braiding:
\begin{proposition}
The map (\ref{ydbraiding}) is a braiding on $M$ if and only if 
$M$ is a Yetter-Drinfel'd module.
\end{proposition}


\begin{bemerkung}
While (\ref{ydcondition2}) is maybe easier to memorise,
(\ref{ydcondition}) makes sense for all bialgebras and is
directly the condition that occurs when testing whether or not 
$ \tau $ satisfies the braid relation. More generally, 
$ \tau $ can be extended to braidings $N \otimes M
\rightarrow M \otimes N$ between any right $H$-module $N$ and a
Yetter-Drinfel'd module $M$, and this identifies the category of
Yetter-Drinfel'd modules with the Drinfel'd centre of the
category of right $H$-modules. 
\end{bemerkung}

\subsection{The Yetter-Drinfel'd module $ \mathrm{ker}\, \varepsilon$}
\label{universal_codifferential}
The following example of a Yetter-Drinfel'd module is
of particular importance to us:

\begin{proposition}
If $H$ is any Hopf algebra, then the kernel 
$ \mathrm{ker}\, \varepsilon $ of its counit is a
Yetter-Drinfel'd module with respect to the right
adjoint action 
$$
	g \bract h:=S(h_{(1)}) g h_{(2)} 
$$
and the right coaction
\begin{equation*}
\tilde{\triangle} : \mathrm{ker}\, \varepsilon \rightarrow 
	\mathrm{ker}\, \varepsilon \otimes H,\quad
	k \mapsto h_{(1)} \otimes h_{(2)} - 1 \otimes h.
\end{equation*}     
\end{proposition}

One can view $ \mathrm{ker}\, \varepsilon$ as a
bicomodule with respect to the trivial left coaction $h
\mapsto 1 \otimes h$, and then the inclusion map $\iota
:\mathrm{ker}\, \varepsilon \rightarrow H$ is a
coderivation. This is universal in the sense that every
coderivation factors through $ \iota $:

\begin{lemma}\label{universalproperty}
Let $H$ be a bialgebra, 
$M$ be an $H$-bicomodule, and $f : M \rightarrow H$ be
a coderivation. 
\begin{enumerate}
\item We have 
$ \mathrm{im}\,f \subseteq \mathrm{ker}\,
\varepsilon $.
\item The restriction of $f$ to $\tilde f : {}^\mathrm{inv}M
\rightarrow \mathrm{ker}\, \varepsilon$ is
right $H$-colinear with respect to the coaction 
$\tilde\triangle$ on $ \mathrm{ker}\, \varepsilon$.
\item If $M$ is a tetramodule and $f$ is $H$-bilinear, then
$\tilde f$ is a morphism of
Yetter-Drinfel'd modules. 
\end{enumerate}
\end{lemma}   
     
\begin{proof}
(1) Applying $ \varepsilon \otimes \varepsilon $ to the
coderivation condition
$$
	(f(m))_{(1)}\otimes(f(m))_{(2)}\,=
	\,m_{(-1)}\otimes f(m_{(0)})
	+m_{(0)}\otimes f(m_{(1)})
$$
yields $\varepsilon(f(m))=2 \varepsilon (f(m))$, so
$ \varepsilon (f(m))=0$. 

(2) For left invariant $m \in M$, we have 
$m_{(-1)} \otimes m_{(0)}=1 \otimes m$, so subtracting 
$1 \otimes f(m)$ from the
coderivation condition yields 
$$
	\tilde \triangle (f(m))\,=\,(f(m))_{(1)}\otimes(f(m))_{(2)}-
	1 \otimes f(m)\,=
	\,
	m_{(0)}\otimes f(m_{(1)}).
$$

(3) The right action on ${}^\mathrm{inv} M$ respectively $
\mathrm{ker}\, \varepsilon $ is obtained from the
bimodule structure on $M$ respectively $H$ by passing
to the right adjoint actions, so 
$\tilde f(m \bract h)=f(S(h_{(1)})mh_{(2)})=S(h_{(1)} 
	f(m) h_{(2)}=\tilde f(m) \bract h$.  
\end{proof}

\begin{bemerkung}\label{fodc}
In Remark~\ref{fodc} we mentioned that
first order bicovariant differential calculi in the
sense of Woronowicz are
formally dual to certain bialgebras in $\LM$. We
can explain this now in more detail:
given a first order
bicovariant differential calculus  
over a Hopf algebra $A$, that is, a bicolinear derivation 
$d : A \rightarrow \Omega$ with values in a
tetramodule $ \Omega $ which is minimal in the
sense that $\Omega=\mathrm{span}_k\{a db \mid a,b \in
A\}$, one defines 
$$
	\mathcal{R}_{(\Omega,d)}:=
	\{a \in \mathrm{ker}\, \varepsilon \mid 
	S(a_{(1)}) da_{(2)}=0\}.
$$ 
It turns out that $(\Omega,d)
\mapsto \mathcal{R}_{(\Omega,d)}$ 
establishes a one-to-one
correspondence between first order bicovariant
differential calculi and right ideals in $
\mathrm{ker}\, \varepsilon $ that are invariant under
the right adjoint coaction $a \mapsto a_{(2)} \otimes
S(a_{(1)})a_{(3)}$ of $A$, see \cite[Proposition~14.1
and Proposition~14.7]{KliSch}. When
$A=k[G]$ is the coordinate ring of an affine algebraic
group, $\Omega$ are the  K\"ahler
differentials and $da$ is the differential of a regular
function $a$, then $\mathcal{R}_{(\Omega,d)}$ is just 
$ (\mathrm{ker}\, \varepsilon)^2$ and 
$ \mathrm{ker}\, \varepsilon /
\mathcal{R}_{(\Omega,d)}$ is the cotangent space of $G$
in the unit element.     

Motivated by this example, one introduces the \emph{quantum
tangent space} 
$$
	\mathcal{T}_{(\Omega,d)}:=\{\phi \in A^* \mid 
\phi (1)=0,\phi(a)=0 \,\forall\,a \in
\mathcal{R}_{(\Omega,d)}\},
$$ 
where  
$A^*=\mathrm{Hom}_k(A,k)$ denotes the dual algebra of
$A$. Provided that $ \Omega $ is finite-dimensional 
in the sense that $\dim_k {}^\mathrm{inv} \Omega <
\infty$, the quantum tangent space  
belongs to the Hopf dual $H:=A^\circ $
of $A$ and uniquely characterises the calculus up to
isomorphism, see \cite[Proposition~14.4]{KliSch} and
the subsequent discussion. By definition,
$\mathcal{T}_{(\Omega,d)}$ is then a subspace of 
$ \mathrm{ker}\, \varepsilon \subset H$ which is by   
\cite[(14)]{KliSch} invariant
under the right coaction $\tilde\triangle$ and 
as a consequence of \cite[Proposition~14.7]{KliSch} it
is also invariant under the right adjoint action of $H$
on itself; in other words, the quantum tangent space is
a Yetter-Drinfel'd submodule of $ \mathrm{ker}\,
\varepsilon $, and if we equip 
$M:=H \otimes \mathcal{T}_{(\Omega,d)}$ with the
corresponding $H$-tetramodule structure we can extend
the inclusion of the quantum tangent space into $
\mathrm{ker}\, \varepsilon $ to a Hopf algebra 
object $f : M \rightarrow H$ in $\LM$. 
Thus first order bicovariant differential
calculi should be viewed as structures dual to Hopf algebra
objects $f : M \rightarrow H$ in $\LM$ for which the induced
map $\tilde f$ is injective.  
\end{bemerkung}

\section{Braided Leibniz algebras}
\label{generalised}

The definition of a Leibniz algebra extends
straightforwardly from ${\tt Vect}$ 
to other additive braided monoidal
categories \cite{Leb}. In this
final section we discuss the construction of such 
generalised Leibniz algebras from Hopf algebra objects in $\LM$
which is the main objective of our paper.   

\subsection{Definition}
The following structure is meant to generalise 
both racks and Leibniz algebras in their role of
domains of objects in $\LM$:

\begin{definition}
A \emph{braided Leibniz algebra} is a  
vector space  $M$ together with linear maps
\begin{equation*}
	\lhd : M \otimes M \rightarrow M,\quad
	x \otimes y \mapsto x \lhd y
\end{equation*}
and
\begin{equation*}
	\tau : M \otimes M \rightarrow M \otimes M,\quad
	x \otimes y \mapsto y_{\langle 1 \rangle } 
	\otimes x_{\langle 2 \rangle }
\end{equation*}
satisfying
\begin{equation}\label{leibniz}
	(x \lhd y) \lhd z=
	x \lhd (y \lhd z)+
	(x \lhd z_{\langle 1 \rangle }) \lhd y_{\langle 2
\rangle }  \quad
	\forall x,y,z \in M.
\end{equation}
\end{definition}
\begin{bemerkung}
We do not assume that $ \tau $ maps elementary tensors 
to elementary tensors, the notation $ y_{\langle 1
\rangle }\otimes x_{\langle 2 \rangle }$ 
should be understood
symbolically like Sweedler's notation $ \triangle(h)=h_{(1)}
\otimes h_{(2)}$ for the coproduct of an element $h$ of a 
coalgebra $H$ which is also in general not an elementary tensor. 
\end{bemerkung}
\begin{bemerkung}
It is natural to ask for $ \tau $ 
to satisfy the braid relation (Yang-Baxter equation), so that
$M$ is just a braided Leibniz algebra as studied
e.g.~in~\cite{Leb}. 
Instead of assuming this a priori we 
rather characterise this 
case in the examples that we study below, and later we
investigate the consequences of this condition. 
\end{bemerkung}

\begin{beispiel}\label{classical}
When  
$\tau$ is the tensor flip, 
$y_{\langle 1 \rangle} 
\otimes x_{\langle 2 \rangle }=y
\otimes x$, we recover 
Definition~\ref{defi_leibni} from Section~\ref{laibniz}  
with $x \lhd y=:[x,y]$, as the Leibniz rule 
(\ref{leibniz}) becomes the (right) Jacobi identity in the form
\begin{equation*}
	[[x,y],z]=
	[x,[y,z]]+
	[[x,z],y].
\end{equation*}
\end{beispiel}

\subsection{Leibniz algebras from modules-comodules}
The following proposition allows one to construct Leibniz
algebras from modules-comodules:
\begin{proposition}
Let $M$ be a right module and a 
right comodule over a bialgebra $H$,
$q : M \rightarrow H$
be a $k$-linear map, and define
\begin{equation*}
	x \lhd y:=xq(y).
\end{equation*}
Then $(M,\tau,\lhd)$ is a braided Leibniz algebra with respect
to 
$$ 
	\tau : M \otimes M \rightarrow M \otimes M,\quad
	x \otimes y \mapsto y_{(0)} \otimes xy_{(1)}
$$ from (\ref{ydbraiding}) provided that
\begin{equation}\label{equiv}
	h_{(1)} q(xh_{(2)})=q(x)h	
\end{equation}
and
\begin{equation}\label{infcoderv}
	q(x)_{(1)} \otimes q(x)_{(2)} =
	1 \otimes q(x)+q(x_{(0)}) \otimes x_{(1)}
\end{equation}
holds for all $x \in M$ and $h \in H$.
\end{proposition}
\begin{proof}
Straightforward computation gives
\begin{eqnarray*}
 (x \lhd y) \lhd z
&=& (xq(y))q(z)=x (q(y)q(z))\\
&=& x (q(z)_{(1)} q(yq(z)_{(2)}))\\
&=& x q(yq(z)) + x q(z_{(0)}) q(yz_{(1)})\\
&=& x \lhd (y \lhd z)+
	(x \lhd z_{\langle 1 \rangle}) \lhd y_{\langle 2
\rangle}
\end{eqnarray*}
as had to be shown.
\end{proof}
\begin{bemerkung}
Observe that applying $ \mathrm{id}_H \otimes \varepsilon $ 
to (\ref{infcoderv}) implies 
\begin{equation*}
	q(x) = \varepsilon (q(x)) + q(x),
\end{equation*}
so this condition necessarily requires 
$ \mathrm{im}\, q \subseteq \mathrm{ker}\, \varepsilon \subset H$. 
If $H$ is a Hopf algebra, then (\ref{equiv}) 
is equivalent to the right $H$-linearity of $q$
with respect to the right adjoint action
of $H$ on $ \mathrm{ker}\, \varepsilon $.
Furthermore, the condition (\ref{infcoderv}) can be stated 
also as saying that $ q : M \rightarrow \mathrm{ker}\,
\varepsilon $ is right $H$-colinear with respect to the
right coaction $ \tilde\triangle $ on $\mathrm{ker}\,
\varepsilon $ from
Section~\ref{universal_codifferential}.
\end{bemerkung}

Thus we can restate the above proposition also as
follows:
   
\begin{corollary}\label{korollar}
Let $M$ be a right module and  
right comodule over a Hopf algebra $H$ and 
$q : M \rightarrow \mathrm{ker}\, \varepsilon$
be an  $H$-linear and $H$-colinear map. Then
\begin{equation*}
	\tau (x \otimes y):=y_{(0)} \otimes xy_{(1)},\quad
	x \lhd y:=xq(y)
\end{equation*}
turns $M$ into a braided Leibniz algebra.
\end{corollary}

\subsection{Leibniz algebras from Hopf algebra objects in
$\LM$}
Altogether, the above results provide a proof of our main theorem:


\begin{proof}[Proof of Theorem~\ref{main}]
From the description of Hopf algebra objects in the category of 
linear maps ${\mathcal L}{\mathcal M}$
in Section \ref{category_LM}, it follows that 
$f : M\to H$ is the data of a Hopf algebra $H$, a
tetramodule $M$ and a morphism of bimodules $f$ 
which is also a coderivation. 
Hence Lemma~\ref{universalproperty} proves the first
part of the theorem. Now  
Corollary \ref{korollar} applied to $q:=\tilde f$ 
yields the structure of 
a braided Leibniz algebra on $\,^{\rm inv}M$.
\end{proof}    

Now we see that classical Leibniz algebras can be viewed as a
special case of the constructions from this subsection:

\begin{beispiel}   \label{example_Leibniz}
Let $(\mathfrak{g},[\cdot,\cdot])$ 
be a (right) Leibniz algebra in the 
category of $k$-vector spaces with the flip as braiding as
in Example~\ref{classical}. We have recalled in Section
\ref{algebraic_objects} how to
regard $\mathfrak{g}$ as a Lie algebra object in ${\mathcal L}{\mathcal M}$, and 
in Section~\ref{univ_env} how to associate
to it its universal enveloping algebra, 
which is a Hopf algebra object 
$\phi : U \mathfrak{g}_\mathrm{Lie} \otimes
\mathfrak{g} \rightarrow U \mathfrak{g}_\mathrm{Lie}$  
in ${\mathcal L}{\mathcal M}$. The  
canonical quotient map 
$\pi : \mathfrak{g}  \rightarrow \mathfrak{g}_{\rm
Lie}$ is given by $ \pi (x)=\phi(1 \otimes x)$.

Recall now from Example~\ref{einhuellende} that 
$ \mathfrak{g} $ is recovered as 
${}^\mathrm{inv} (U \mathfrak{g}_\mathrm{Lie} \otimes
\mathfrak{g})$ (with trivial right coaction), and in
this sense, $ \pi $ coincides with $\tilde \phi$. The 
Yetter-Drinfel'd braiding thus becomes the tensor
flip, and the generalised Leibniz bracket 
$\lhd$ on $\mathfrak{g}$ is the original one.  

This generalises the corresponding example for Lie algebras \cite{Maj}  p. 63, \cite{CCES} 
Proposition 3.5, to Leibniz algebras.  
\end{beispiel}

The above example should be viewed as an infinitesimal 
variant of the following one: 

\begin{beispiel}\label{beispielb}
Let $X$ be a finite rack and $G:=\mathrm{As}(X)$ be
its associated group \cite{FenRou}. Then $p : X\to G$ is
an augmented rack, see Remark~\ref{associated_group}
above. 
We have seen in
Proposition \ref{prop_aug_rack} that the linearisation $p:kX\to kG$ 
is not a Hopf algebra object in
${\mathcal L}{\mathcal M}$, so we cannot apply 
Theorem~\ref{main} in this situation in order to obtain a
Leibniz algebra structure on $kX$.  

However, recall from Example~\ref{augmented_rack} that 
$kX$ is by the very
definition of an augmented rack a
Yetter-Drinfel'd module over the group algebra $kG$,
and we obtain a
morphism $q : kX \rightarrow \mathrm{ker}\, \varepsilon
\subset kG$, $x \mapsto p(x)-1$ of
Yetter-Drinfel'd modules. Now we can apply
Corollary~\ref{korollar} to obtain a braided Leibniz
algebra structure $x \lhd y=x(p(y)-1)$. 
This construction works for all augmented racks, so
augmented racks can be converted into special examples
of braided Leibniz algebras.
In this way, we recover \cite[Proposition~3.1]{CCES}.  
\end{beispiel}

\begin{example}\label{fodcb}
If $\mathcal{T} \subset H:=A^\circ$ is the quantum tangent
space of a finite-dimensional first order 
bicovariant differential calculus over a Hopf algebra
$A$ 
and $f : H \otimes \mathcal{T} \rightarrow H$ is the
corresponding Hopf algebra object in $\LM$
(recall Remark~\ref{fodca}), then the generalised
Leibniz  
bracket from Theorem~\ref{main} becomes  
$$
	x \lhd y=x \tilde f(y)=S(y_{(1)})xy_{(2)}.
$$ 
That is, the generalised Leibniz algebra structure is
precisely the quantum Lie algebra structure 
of $\mathcal{T}$, compare
\cite[Section~14.2.3]{KliSch}.
\end{example}
\begin{example}
We end by explicitly computing the R-matrix
representing the Yetter-Drinfel'd
braiding from Example \ref{example_Leibniz} for the 
Hei\-sen\-berg-Voros algebra ${\mathfrak g}$. 
This is the $3$-dimensional Leibniz algebra spanned by $x,y,z$ 
such that the only non-trivial brackets are
$$
	[x,x]\,=\,z,\,\,\,\,\,\,\,\,[y,y]\,=\,z,\,\,\,\,\,\,\,[x,y]\,=\,z,\,\,\,\,\,\,\,[y,x]\,=\,-z
$$

This Leibniz algebra can also be described as a $1$-dimensional central extension of the 
abelian $2$-dimensional Lie/Leibniz algebra, but rather than being antisymmetric, the cocycle 
has a symmetric and an antisymmetric part (in contrast to the Heisenberg Lie algebra).  

The shelf structure on ${\mathfrak g}$ is given for constants $a,b,c,d,a',b',c',d'\in k$ by
\begin{eqnarray}
&& (a+bx+cy+dz)\lhd(a'+b'x+c'y+d'z)\nonumber\\
&=& aa'+a'bx+a'cy+z(a'd+bb'+bc'-cb'+cc').\nonumber
\end{eqnarray}

One computes the R-matrix to be\\

$$\left(\begin{array}{cccccccccccccccc} 
1 & 0 & 0 & 0 & 0 & 0 & 0 & 0 & 0 & 0 & 0 & 0 & 0 & 0 & 0 & 0 \\
0 & 0 & 0 & 0 & 1 & 0 & 0 & 0 & 0 & 0 & 0 & 0 & 0 & 0 & 0 & 0 \\
0 & 0 & 0 & 0 & 0 & 0 & 0 & 0 & 1 & 0 & 0 & 0 & 0 & 0 & 0 & 0 \\
0 & 0 & 0 & 0 & 0 & 0 & 0 & 0 & 0 & 0 & 0 & 0 & 1 & 0 & 0 & 0 \\
0 & 1 & 0 & 0 & 0 & 0 & 0 & 0 & 0 & 0 & 0 & 0 & 0 & 0 & 0 & 0 \\
0 & 0 & 0 & 0 & 0 & 1 & 0 & 0 & 0 & 0 & 0 & 0 & 0 & 0 & 0 & 0 \\
0 & 0 & 0 & 0 & 0 & 0 & 0 & 0 & 0 & 1 & 0 & 0 & 0 & 0 & 0 & 0 \\
0 & 0 & 0 & 0 & 0 & 0 & 0 & 0 & 0 & 0 & 0 & 0 & 0 & 1 & 0 & 0 \\
0 & 0 & 1 & 0 & 0 & 0 & 0 & 0 & 0 & 0 & 0 & 0 & 0 & 0 & 0 & 0 \\
0 & 0 & 0 & 0 & 0 & 0 & 1 & 0 & 0 & 0 & 0 & 0 & 0 & 0 & 0 & 0 \\
0 & 0 & 0 & 0 & 0 & 0 & 0 & 0 & 0 & 0 & 1 & 0 & 0 & 0 & 0 & 0 \\
0 & 0 & 0 & 0 & 0 & 0 & 0 & 0 & 0 & 0 & 0 & 0 & 0 & 0 & 1 & 0 \\
0 & 0 & 0 & 1 & 0 & 1 &-1 & 0 & 0 & 1 & 1 & 0 & 0 & 0 & 0 & 0 \\
0 & 0 & 0 & 0 & 0 & 0 & 0 & 1 & 0 & 0 & 0 & 0 & 0 & 0 & 0 & 0 \\
0 & 0 & 0 & 0 & 0 & 0 & 0 & 0 & 0 & 0 & 0 & 1 & 0 & 0 & 0 & 0 \\
0 & 0 & 0 & 0 & 0 & 0 & 0 & 0 & 0 & 0 & 0 & 0 & 0 & 0 & 0 & 1 
\end{array}\right)$$\\
 
Observe the 13th line. This matrix does not square to $1$. 
\end{example}





\end{document}